\documentclass{amsart}
\usepackage{amsmath,amsthm,amssymb,amscd,amsfonts}
\newtheorem{theorem}{Theorem}[section]

\newtheorem{proposition}[theorem]{Proposition}
\newtheorem{corollary}[theorem]{Corollary}

\theoremstyle{definition}
\newtheorem{definition}[theorem]{Definition}
\newtheorem{example}[theorem]{Example}

\newcommand{\Span}{\operatorname{span}}

\begin{document}

\title[New Characterizations of Fusion Bases and Riesz Fusion Bases]
{New Characterizations of Fusion Bases and\\Riesz Fusion Bases in Hilbert Spaces}

\author[M. S. Asgari]{Mohammad Sadegh Asgari}
\address{Department of Mathematics, Faculty of Science,
Islamic Azad University, Central Tehran Branch,
 P.O.Box 13185/768, Tehran,Iran.}
\curraddr{}
\email{msasgari@yahoo.com \;\;\;\; moh.asgari@iauctb.ac.ir}
\thanks{}
\thanks{}
\subjclass[2000]{Primary 42C15; Secondary 46C99.}
\keywords{Fusion Frame; Riesz fusion basis; Exact fusion frame; Orthonormal fusion basis;
Besselian fusion frame.}
\date{\today}
\dedicatory{Department of Mathematics, Faculty of Science,\\Islamic Azad University, Central Tehran Branch,
 Tehran, Iran.\\moh.asgari@iauctb.ac.ir ; msasgari@yahoo.com}

\begin{abstract}

In this paper we investigate a new notion of bases in Hilbert spaces and similar to fusion frame theory
we introduce fusion bases theory in Hilbert spaces. We also introduce a new definition of fusion dual
sequence associated with a fusion basis and show that the operators of a fusion dual sequence are
continuous projections. Next we define the fusion biorthogonal sequence, Bessel fusion basis, Hilbert
fusion basis and obtain some characterizations of them. we study orthonormal fusion systems and Riesz
fusion bases for Hilbert spaces. we consider the stability of fusion bases under small perturbations.
We also generalized a result of Paley-Wiener [13] to the situation of fusion basis.
\end{abstract}

\maketitle

\section{Introduction}

Frames for Hilbert spaces were first formally defined by Duffin and Schaeffer [6] in 1952
to study some deep problems in nonharmonic Fourier series,
reintroduced in 1986 by Daubechies, Grossmann and Meyer [5] and popularized from
then on. A frame is a redundant set of vectors in a Hilbert space with the property that provide
usually non-unique representations of vectors in terms of the frame elements.
Fusion frames which were considered recently as generalized frames in Hilbert spaces,
were introduced by Casazza and Kutyniok in [3] and have quickly turned into an industry.
Related approaches with a different focus were undertaken by M.Fornasier [7] and W.Sun [12], D. R. Larson et al.[10]. Bases, frames
and fusion frames play important roles in many applications in mathematics, science, and engineering, including coding theory,
filter bank theory, sigma-delta quantization, signal and image processing and wireless communications and many other areas.\par
The main subject of this paper deals with fusion bases and resolution of the identity. The paper is organized as follows: Section 2,
contains a new definition of fusion basis in a Hilbert space. In this section similar to basis theory we first establishes a simple
criterion for determining when a complete set of closed subspaces is a fusion basis. Next we introduce the concepts of fusion biorthogonal
sequence, Bessel fusion basis, Hilbert fusion basis and obtain some characterizations of them. In Section 3, we study orthonormal fusion bases
and Riesz fusion bases for Hilbert spaces. We introduce a new definition of Riesz fusion basis and then give some characterizations of
orthonormal fusion bases and Riesz fusion bases. In Section 4, we study the stability of fusion bases under small perturbations.
we also generalized a result of Paley-Wiener [13] to the situation of fusion basis.\par
Throughout this paper, $\mathcal{H},\mathcal{K}$ are separable Hilbert spaces and $I, I_j, J$ denote the countable (or finite) index sets and
$\pi_{W}$ denote the orthogonal projection of a closed subspace $W$ of $\mathcal{H}$. We will always use $\mathcal{R}_T$ and $\mathcal{N}_T$
to denote range and the null spaces of an operator $T\in B(\mathcal{H},\mathcal{K})$ respectively.\par Let $\mathcal{W}=\{W_j\}_{j\in J}$ be a
sequence of closed subspaces in $\mathcal{H}$, and let $\mathcal{A}=\{\alpha_j\}_{j\in J}$ be a family of weights, i.e., $\alpha_j >0 $ for all
$j\in J$. A sequence $\mathcal{W}_{\alpha}=\{(W_j,\alpha_j)\}_{j\in J}$ is a fusion frame, if there exist real numbers $0<C\leq D<\infty$ such
that,\begin{equation} C\|f\|^2 \leq\sum_{j\in J}\alpha_j^2 \|\pi_{W_j} (f)\|^2 \leq D\|f\|^2 \;\;\;\;\; \text{for all $f\in \mathcal{H}$}.
\end{equation} The constant $C,D$ are called the fusion frame bounds. If $C=D=\lambda,$ the fusion frame is $\lambda$-tight and it is a Parseval
fusion frame if $C=D=1$, and it is $\alpha$-uniform if $\alpha=\alpha_i=\alpha_j$ for all $i,j\in J$. If the right-hand inequality of $(1)$ holds,
then we say that $\mathcal{W}_{\alpha}$ is a Bessel fusion sequence with Bessel fusion bound $D$.\par For each sequence $\{W_j\}_{j\in J}$ of
closed subspaces of $\mathcal{H}$, we define the Hilbert space associated with $\mathcal{W}$ by \begin{equation} \Big(\sum_{j\in J}\oplus W_j\Big)
_{\ell^2}=\Big\{ \{g_j\}_{j\in J} | g_j\in W_j  \;\; \text{and}\;\;\sum_{j\in J} \|g_j\|^2 < \infty \Big\}. \end{equation} with inner product given by
\begin{equation} <\{f_k\}_{k\in J}, \{g_k\}_{k\in J} >=\sum_{j\in J} <f_j, g_j>.\end{equation} For more details about the theory and application
of bases, frames and fusion frames we refer the reader to the books by Young [13], Christensen [4], the survey articles by Asgari [1,2], Casazza [3],
Gavruta [8], and Holub [9].

\section{Fusion Schauder bases}

The concept of Riesz decomposition that we call fusion Schauder basis, was first introduced by Casazza and Kutyniok in [3]. In this section,
we develop the fusion basis theory for Hilbert spaces. As a consequence we generalized some results of bases to fusion bases.
\begin{definition} Let $\mathcal{W}=\{W_j\}_{j\in J}$ be a sequence of closed subspaces of $\mathcal{H}$, then $\mathcal{W}=\{W_j\}_{j\in J}$
is called a fusion Schauder basis or simply a $f$-basis for $\mathcal{H}$ if for any $f\in \mathcal{H}$ there exists an unique sequence
$\{g_j:\;g_j\in W_j\}_{j\in J}$ such that\begin{equation} f=\sum_{j\in J}g_j,\end{equation} with the convergence being in norm. If the
series (4) converges unconditionally for each $f\in \mathcal{H},$ we say that $\mathcal{W}$ is an unconditional $f$-basis.\end{definition}
\begin{example} For each $N\in\mathbb{N}$, let $\mathcal{H}=\mathbb{C}^{N}$ and let $\{e_i\}_{i=1}^{N}$ be the standard basis of $\mathbb{C}^{N}$.
If the subspace $W_j\subset\mathcal{H}$ defined by $$W_j=\Span\{\sum_{i=1, i\neq j}^{N}e_i\},$$ for all $1\leq j\leq N$. Then $\{W_j\}_{j=1}^{N}$
is a $f$-basis for $\mathcal{H}$.\end{example}
\begin{example} Let $\mathcal{H}=\ell^2(\mathbb{N})$ and let, $\{e_i\}_{i\in\mathbb{N}}$ be the standard basis of $\ell^2(\mathbb{N})$.
For each $j\in\mathbb{N}$ define the subspace $W_j\subset\mathcal{H}$ by $W_j=\Span\{e_{2j-1}, e_{2j}\}$. Then $\{W_j\}_{j\in\mathbb{N}}$ is a
$f$-basis for $\mathcal{H}$, because if $f=\sum_{j\in\mathbb{N}}\lambda_{j}e_{2j-1}+\mu_{j}e_{2j}$ then it is easy to check
that $\lambda_{j}=<f, e_{2j-1}>$ and $\mu_j=<f, e_{2j}>$.\end{example}
\begin{theorem} Let $\{W_j\}_{j\in J}$ be a $f$-basis for $\mathcal{H}$. Then $\dim\mathcal{H}=\sum_{j\in J}\dim W_j$.\end{theorem}
\begin{proof} Let $\{e_{ij}\}_{i\in J_j}$ be an orthonormal basis for $W_j$ for all $j\in J$. We show that $\{e_{ij}\}_{j\in J, i\in J_j}$
is a basis for $\mathcal{H}$. Since $\{e_{ij}\}_{i\in J_j}$ is an orthonormal basis for $W_j$, hence every $g_j\in W_j$ has a unique
expansion of the form $g_j=\sum_{i\in J_j}<g_j, e_{ij}>e_{ij}$. This implies that also every $f\in\mathcal{H}$ has a unique expansion
of the form $$f=\sum_{j\in J}\sum_{i\in J_j}<g_j, e_{ij}>e_{ij}.$$ This shows that $\dim\mathcal{H}=\sum_{j\in J}\dim W_j$.\end{proof}
\begin{corollary} Let $\{W_j\}_{j\in J}, \{V_i\}_{i\in I}$ be $f$-bases for $\mathcal{H}$. Then $\sum_{j\in J}\dim W_j=\sum_{i\in I}\dim V_i$.
\end{corollary} Let $\{W_j\}_{j\in J}$ be a $f$-basis for $\mathcal{H}$, then every $f\in\mathcal{H}$ has a unique
expansion of the form $f=\sum_{j\in J}g_j$. Hence it is clear that each $g_j\in W_j$ is a linear operator of $f$. If we denote this linear
operator by $P_{W_j}:\mathcal{H}\rightarrow W_j$, then $g_j=P_{W_j}f$, and we have $f=\sum_{j\in J}P_{W_j}f$. The sequence $\{P_{W_j}\}_{j\in J}$
is called the $f$-dual sequence of $\{W_j\}_{j\in J}$ and $\mathcal{W}=\{(W_j, P_{W_j})\}_{j\in J}$ is called $f$-basis system.\par
In the next theorem we show that the operators of a $f$-dual sequence are continuous projections.
\begin{theorem} Let $\mathcal{W}=\{W_j\}_{j\in J}$ be a $f$-basis for $\mathcal{H}$, with $f$-dual sequence $\{P_{W_j}\}_{j\in J}$. Then
$$P_{W_j}\in B(\mathcal{H}, W_j)\;\;\;\;\; \text{and}\;\;\;\;\;\;\; P_{W_i}P_{W_j}=\delta_{ij}P_{W_j} \;\;\;\;\;\forall i,j\in J,$$
where $\delta_{ij}$ is the Kronecker delta.\end{theorem}\begin{proof} Define the space $$\mathcal{A}=\Big\{\{g_j\}_{j\in J}|\;g_j\in W_j,
\;\;\sum_{j\in J}g_j \;\;\;\text{is convergent}\Big\},$$ with the norm defined by $$\big\|\{g_j\}_{j\in J}\big\|=\sup_{0<|F|<\infty\atop{F\subseteq J}}
\Big\|\sum_{i\in F}g_i\Big\|<\infty.$$ It is clear that $\mathcal{A}$ endowed with this norm, is a normed space with respect to
the pointwise operations. We show that the space $\mathcal{A}$ is a complete. Let $\{u_n\}_{n\in\mathbb{N}}$ be a Cauchy sequence
in $\mathcal{A}$. If $u_n=\{g_{nj}\}_{j\in J}$, then given any $\varepsilon>0$, there exists a number $N$ such that
\begin{equation}\sup_{0<|F|<\infty\atop{F\subseteq J}}\Big\|\sum_{i\in F}(g_{ni}-g_{mi})\Big\|<\varepsilon\end{equation} for all $m,n\geq N$.
This yields\begin{align*} \|g_{nj}-g_{mj}\|\leq\sup_{0<|F|<\infty\atop{F\subseteq J}}\Big\|\sum_{i\in F}(g_{ni}-g_{mi})\Big\|<\varepsilon,
\end{align*} for all $j\in J$ and $m,n\geq N$. It follows that $\{g_{nj}\}_{n\in\mathbb{N}}$ is a Cauchy sequence in $W_j$ and thus convergent.
Let $g_j\in W_j$ such that $g_j=\lim_{n\to\infty}g_{nj}$ and $u=\{g_j\}_{j\in J}$. From (5), by letting $m\to\infty$, we obtain \begin{equation}
\sup_{0<|F|<\infty\atop{F\subseteq J}}\Big\|\sum_{i\in F}(g_{ni}-g_i)\Big\|\leq\varepsilon\end{equation} for all $n\geq N$.
Moreover, for every finite subset $F\subset J$ we have\begin{align*}\Big\|\sum_{i\in F}g_i\Big\|&\leq\Big\|\sum_{i\in F}
(g_{Ni}-g_i)\Big\|+\Big\|\sum_{i\in F}g_{Ni}\Big\|\\&\leq\sup_{0<|F|<\infty\atop{F\subseteq J}}\Big\|\sum_{i\in F}(g_{Ni}-g_{i})\Big\|
+\sup_{0<|F|<\infty\atop{F\subseteq J}}\Big\|\sum_{i\in F}g_{Ni}\Big\|\end{align*} which implies that $u\in\mathcal{A}$. Further (6) implies that
the sequence $\{u_n\}_{n\in\mathbb{N}}$ is convergent to $u$ in $\mathcal{A}$. This proves that $\mathcal{A}$ is a Banach space.
Now define the mapping $$T:\mathcal{A}\rightarrow\mathcal{H}\;\;\;\;\text{with}\;\;\;\;T(\{g_j\}_{j\in J})=\sum_{j\in J}g_j.$$
Since $\mathcal{W}$ is a $f$-basis for $\mathcal{H}$ hence $T$ is linear, one-to-one and onto. On the other hand, since $$\|T(\{g_j\}_{j\in J})\|
=\Big\|\sum_{j\in J}g_j\Big\|\leq\sup_{0<|F|<\infty\atop{F\subseteq J}}\Big\|\sum_{i\in F}g_i\Big\|=\|\{g_j\}_{j\in J}\|.$$ Thus $T$ is continuous
and by open mapping theorem $T^{-1}$ is also continuous. This shows that $\mathcal{A}$ and $\mathcal{H}$ are Banach spaces isomorphic.
Now suppose that $f=\sum_{j\in J}g_j$ is a fixed, arbitrary element of $\mathcal{H}$ and let $j\in J$ be arbitrary. Then we have
\begin{align*}\|P_{W_j}f\|&=\|g_j\|\leq\sup_{0<|F|<\infty\atop{F\subseteq J}}\big\|\sum_{i\in F}g_i\big\|=\|T^{-1}f\|\leq\|T^{-1}\|\|f\|.
\end{align*} This shows that each $P_{W_j}$ is continuous and $\|P_{W_j}\|\leq\|T^{-1}\|$. Finally from $P_{W_i}g_j=\delta_{ij}g_j$ we have $P_{W_i}P_{W_j}
=\delta_{ij}P_{W_j}$ for all $i,j\in J$.\end{proof} Let $\{W_j\}_{j\in J}$ be a $f$-basis for $\mathcal{H}$ and let $\{P_{W_j}\}_{j\in J}$ be the $f$-dual
sequence of $\{W_j\}_{j\in J}$. Then $F$-partial sum operator of $\{W_j\}_{j\in J}$ defined by
$$S_{F}:\mathcal{H}\rightarrow\mathcal{H}\;\;\;\;\text{with}\;\;\;\;S_{F}f=\sum_{j\in F}P_{W_j}f,$$ for all finite subset $F\subset J$.
By Theorem 2.6, $S_F$ is a bounded operator and \begin{equation} 1\leq\sup_{0<|F|<\infty\atop{F\subseteq J}}\|S_F\|<\infty.\end{equation}
\begin{definition} Let $\{W_j\}_{j\in J}$ be a sequence of closed subspaces of $\mathcal{H}$. Then\begin{itemize}\item [$(i)$] $\{W_j\}_{j\in J}$
is called a complete set for $\mathcal{H}$, if $\mathcal{H}=\overline{\Span}\{W_j\}_{j\in J}$.\item [$(ii)$] A family of operators
$\{Q_j\in B(\mathcal{H}, W_j):\;\; j\in J\}$ is called a $f$-biorthogonal sequence of $\{W_j\}_{j\in J}$, if $Q_ig_j=\delta_{ij}g_j$ for all $i,j\in J,
\;g_j\in W_j$.\end{itemize} As a direct consequence of definition, $\{W_j\}_{j\in J}$ is a complete set for $\mathcal{H}$, if and only if
$$\{f:\;\;\pi_{W_j}f=0, j\in J\}=\{0\}.$$ Moreover it follows from the definition that if $\{Q_j\}_{j\in J}$ is a $f$-biorthogonal sequence of
$\{W_j\}_{j\in J}$ then every $Q_j$ is a projection from $\mathcal{H}$ onto $W_j$ and $Q_i\pi_{W_j}=\pi_{W_j}Q_i^*$.\end{definition}
The following theorem establishes a simple criterion for determining when a complete set is a $f$-basis.
\begin{theorem} A complete sequence of closed subspaces $\{W_j\}_{j\in J}$ of $\mathcal{H}$ is a $f$-basis for $\mathcal{H}$ if and only if
there exists a constant $M$ such that \begin{equation}\Big\|\sum_{j\in F}g_j\Big\|\leq M\Big\|\sum_{j\in G}g_j\Big\|\end{equation}
for all finite subsets $F\subset G\subset J$ and arbitrary vectors $g_j\in W_j$.\end{theorem}\begin{proof} First suppose that $\{W_j\}_{j\in J}$
is a $f$-basis for $\mathcal{H}$ and let $M=\sup_{0<|F|<\infty\atop{F\subseteq J}}\|S_F\|$, then for all finite subsets $F\subset
G\subset J$ and arbitrary vectors $g_j\in W_j$ we have \begin{align*}\big\|\sum_{j\in F}g_j\big\|=\big\|S_{F}\big(\sum_{j\in G}
g_j\big)\big\|\leq M\big\|\sum_{j\in G}g_j\big\|.\end{align*} To prove the opposite implication, let $f\in\mathcal{H}$. By hypothesis, there exist
finite subsets $F_n\subset F_{n+1}\subset J$ and vectors $g_{nj}\in W_j$ for all $n\in\mathbb{N}$ such that $f=\lim_{n\to\infty}\sum_{j\in F_n}g_{nj}$.
For notational convenience, put $g_{nj}=0$ for $j\not\in F_n$, then for every $m>n$ and $j\in F_n$ we have\begin{align*}\|g_{nj}-g_{mj}\|&\leq M
\big\|\sum_{i\in F_n}(g_{ni}-g_{mi})\big\|\\&\leq M^2\big\|\sum_{i\in F_m}(g_{ni}-g_{mi})\big\|\\&=M^2\big\|\sum_{i\in F_n}g_{ni}-\sum_{i\in F_m}g_{mi}
\big\|\rightarrow 0\;\;\;\;(n\rightarrow\infty).\end{align*} This shows that $\{g_{nj}\}_{n\in\mathbb{N}}$ is a Cauchy sequence in $W_j$ and
thus convergent. Let $g_j=\lim_{n\to\infty}g_{nj}$ for some $g_j\in W_j$, then we have $$f=\lim_{n\to\infty}\sum_{j\in F_n}g_{nj}=\sum_{j\in J}g_j.$$
Now we show that this representation is unique. If $\sum_{j\in J}g_j=0$, then for every finite subset $F\subset J$ and $j\in F$ again
from hypothesis, we have $$\|g_j\|\leq M\big\|\sum_{i\in F}g_i\big\|\rightarrow0.$$ This shows that $g_j=0$.
Therefore $\{W_j\}_{j\in J}$ is a $f$-basis for $\mathcal{H}$.\end{proof} Suppose that $\{(W_j, P_{W_j})\}_{j\in J}$ is a $f$-basis system for
$\mathcal{H}$. Then $\{P_{W_j}^*(W_j)\}_{j\in J}$ is a family of closed subspaces of $\mathcal{H}$. The following theorem shows
that $\{(P_{W_j}^*(W_j), P_{W_j}^*)\}_{j\in J}$ is also a $f$-basis system for $\mathcal{H}$.
\begin{theorem} If $\{(W_j, P_{W_j})\}_{j\in J}$ is a $f$-basis system for
$\mathcal{H}$, then $\{(P_{W_j}^*(W_j), P_{W_j}^*)\}_{j\in J}$ is also a $f$-basis system for $\mathcal{H}$.\end{theorem}
\begin{proof} First we prove that $\mathcal{H}=\overline{\Span}\{P_{W_j}^*(W_j)\}_{j\in J}$. To see this,
let $f\perp\overline{\Span}\{P_{W_j}^*(W_j)\}_{j\in J}$. Then $$\|P_{W_j}f\|^2=<f, P_{W_j}^*P_{W_j}f>=0,$$ which implies that $P_{W_j}f=0$ for all $j\in J$.
We also have $f=\sum_{j\in J}P_{W_j}f=0$, hence $\mathcal{H}=\overline{\Span}\{P_{W_j}^*(W_j)\}_{j\in J}$. This shows that every $f\in\mathcal{H}$ has at
least one representation of the form $f=\sum_{j\in J}P_{W_j}^*g_j$ for some sequence $\{g_j:\;\;g_j\in W_j\}_{j\in J}$. We show that this representation
is unique. Assume that $\sum_{j\in J}P_{W_j}^*g_j=0$. Then we have $$P_{W_i}^*g_i=P_{W_i}^*\big(\sum_{j\in J}P_{W_j}^*g_j\big)=0,$$
for all $i\in J$. Therefore $\{P_{W_j}^*(W_j)\}_{j\in J}$ is a $f$-basis for $\mathcal{H}$. Since $P_{W_i}^*P_{W_j}^*=
\delta_{ij}P_{W_j}^*$ for all $i,j\in J$ hence $\{P_j^*\}_{j\in J}$ is the $f$-dual sequence of $\{P_{W_j}^*(W_j)\}_{j\in J}$.
\end{proof}\begin{proposition} Every $f$-basis for a Hilbert space possesses a unique $f$-biorthogonal sequence.
\end{proposition}\begin{proof} Let $\{W_j\}_{j\in J}$ be a $f$-basis for $\mathcal{H}$ with $f$-dual sequence $\{P_{W_j}\}_{j\in J}$.
By definition, $\{P_{W_j}\}_{j\in J}$ is a $f$-biorthogonal sequence of $\{W_j\}_{j\in J}$. Moreover, if $\{Q_j\}_{j\in J}$ is another $f$-biorthogonal
sequence of $\{W_j\}_{j\in J}$, then for any $f\in\mathcal{H}$ and $i\in J$ we have\begin{align*} Q_if=\sum_{j\in J} Q_iP_{W_j}f=\sum_{j\in J}
\delta_{ij}P_{W_j}f=P_{W_i}f.\end{align*} Hence $Q_i=P_{W_i}.$\end{proof}\begin{definition} Let $\{W_j\}_{j\in J}$ be a $f$-basis for $\mathcal{H}$. We say that
$\{W_j\}_{j\in J}$ is a Bessel $f$-basis if whenever $\sum_{j\in J}g_j$ converges, then $\{g_j\}_{j\in J}\in\big(\sum_{j\in J}\oplus W_j\big)_{\ell^2}$.
It is called a Hilbert $f$-basis, if the series $\sum_{j\in J}g_j$ is convergent for all $\{g_j\}_{j\in J}\in\big(\sum_{j\in J}\oplus W_j\big)_{\ell^2}$.
\end{definition}\begin{theorem} $\{W_j\}_{j\in J}$ is a Bessel $f$-basis for $\mathcal{H}$ if and only if there exists a constant $A>0$ such that
$$A\sum_{j\in F}\|g_j\|^2\leq\big\|\sum_{j\in F}g_j\big\|^2$$ for any finite subset $F\subset J$ and any arbitrary vectors $g_j\in W_j$.\end{theorem}
\begin{proof} The sufficiency is trivial. Assume that $\{W_j\}_{j\in J}$ is a Bessel
$f$-basis and consider the space $$\mathcal{A}=\Big\{\{g_j\}_{j\in J}|\;g_j\in W_j,\;\;\sum_{j\in J}g_j \;\;\;\text{is convergent}\Big\}.$$
By assumption $\mathcal{A}$ is a subspace of $\big(\sum_{j\in J}\oplus W_j\big)_{\ell^2}$. We show that $\mathcal{A}$ is closed.
To see this, let $\{g_{nj}\}_{j\in J}$ be a sequence in $\mathcal{A}$ such that converges to some $\{g_j\}_{j\in J}\in\big(\sum_{j\in J}\oplus W_j\big)
_{\ell^2}$, then $g_{nj}\rightarrow g_j$ for all $j\in J$. Let $F\subset J$ be an arbitrary finite subset and let $n\in\mathbb{N}$, then we have
\begin{align*}\big\|\sum_{j\in F}g_j\big\|\leq\big\|\sum_{j\in F}(g_{nj}-g_j)\big\|+\big\|\sum_{j\in F}g_{nj}\big\|.\end{align*} It follows that
$\sum_{j\in J}g_j$ is Cauchy and hence convergent in $\mathcal{H}$, which implies that $\mathcal{A}$ is closed. Now define the operator $T:\mathcal{A}
\rightarrow\mathcal{H}$ by $$T(\{g_j\}_{j\in J})=\sum_{j\in J}g_j.$$ Then, it is obvious that $T$ is linear, one-to-one. To show that $T$ is a bounded,
consider the sequence of bounded linear operators $$T_{F}:\mathcal{A}\rightarrow\mathcal{H},\;\;\;\;\;\;T_{F}(\{g_j\}_{j\in J})=\sum_{j\in F}g_j$$
for all finite subset $F\subset J$. Clearly $T_{F}\rightarrow T$ pointwise, so $T$ is bounded by Banach-Steinhaus Theorem. Now by Theorems 4.13 and
4.15 of [11] there exists a constant $A>0$ such that $$A\sum_{j\in F}\|g_j\|^2\leq\big\|\sum_{j\in F}g_j\big\|^2$$ for any finite subset $F\subset J$
and any arbitrary vectors $g_j\in W_j$.\end{proof}
\begin{theorem} $\{W_j\}_{j\in J}$ is a Hilbert $f$-basis for $\mathcal{H}$, if and only if there exists a constant $B>0$ such that $$\big\|\sum_{j\in F}
g_j\big\|^2\leq B\sum_{j\in F}\|g_j\|^2$$ for any finite subset $F\subset J$ and any arbitrary vectors $g_j\in W_j$.
\end{theorem}\begin{proof} Suppose that $\{W_j\}_{j\in J}$ is a Hilbert $f$-basis then the Banach-Steinhaus Theorem guarantees that the operator
$T:\big(\sum_{j\in J}\oplus W_j\big)_{\ell^2}\rightarrow\mathcal{H}$ defined by $T(\{g_j\}_{j\in J})=\sum_{j\in J}g_j$ is bounded.
Therefore there exists a constant $B>0$ such that $$\big\|\sum_{j\in F}g_j\big\|^2\leq B\sum_{j\in F}\|g_j\|^2$$ for any finite subset
$F\subset J$ and any arbitrary vectors $g_j\in W_j$. The opposite conclusion is trivial.\end{proof}
\begin{theorem} $\{(W_j, P_{W_j})\}_{j\in J}$ is a Bessel $f$-basis system if and only if $\{(P_j^*(W_j), P_{W_j}^*)\}_{j\in J}$
is a Hilbert $f$-basis system.\end{theorem}\begin{proof} Fix $F\subset J$ with $|F|<\infty$. First suppose that $\{W_j\}_{j\in J}$
is a Bessel $f$-basis, then $\{P_{W_j}f\}_{j\in J}\in\big(\sum_{j\in J}\oplus W_j\big)_{\ell^2}$ for all $f\in\mathcal{H}$. If
$f=\sum_{j\in F}P_{W_j}^*g_j$ for some vectors $g_j\in W_j$. Then we have\begin{align*}\big\|\sum_{j\in F}P_{W_j}^*g_j\big\|^4&=
|<f, \sum_{j\in F}P_{W_j}^*g_j>|^2\leq\big(\sum_{j\in F}\|P_{W_j}f\|\|P_{W_j}^*g_j\|\big)^2\\&\leq\big(\sum_{j\in J}\|P_{W_j}f\|^2\big)
\big(\sum_{j\in F}\|P_{W_j}^*g_j\|^2\big).\end{align*}
This shows that $\{P_{W_j}^*(W_j)\}_{j\in J}$ is a Hilbert $g$-basis. For the other implication, assume that $\{P_{W_j}^*(W_j)\}_{j\in J}$
is a Hilbert $f$-basis. If $f=\sum_{j\in F}g_j$ for some vectors $g_j\in W_j$, then $g_j=P_{W_j}f$ for all $j\in F$.
By Theorem 2.6 and Theorem 2.13 there exists a constant $B>0$ such that \begin{align*} \big\|\sum_{j\in F}P_{W_j}^*P_{W_j}f\big\|^2&\leq
B\sum_{j\in F}\|P_{W_j}f\|^2=B<f, \sum_{j\in F}P_{W_j}^*P_{W_j}f>\\&\leq B\|f\|\big\|\sum_{j\in F}P_{W_j}^*P_{W_j}f\big\|.\end{align*}
Hence, $$\big\|\sum_{j\in F}P_{W_j}^*P_{W_j}f\big\|\leq B\big\|\sum_{j\in F}g_j\big\|.$$ We also have \begin{align*}
\sum_{j\in F}\|g_j\|^2&=\sum_{j\in F}\|P_{W_j}f\|^2=<f, \sum_{j\in F}P_{W_j}^*P_{W_j}f>\\&\leq\big\|\sum_{j\in F}g_j\big\|\big\|\sum_{j\in F}
P_{W_j}^*P_{W_j}f\big\|\leq B\big\|\sum_{j\in F}g_j\big\|^2.\end{align*} Now applying Theorem 2.12 the result follows at once.\end{proof}

\section{Orthonormal Fusion Bases and Riesz Fusion Bases}

In this section, we develop a theory of orthonormal fusion bases and Riesz fusion bases for the Hilbert spaces.
\begin{definition} Let $\mathcal{W}=\{W_j\}_{j\in J}$ be a sequence of closed subspaces of $\mathcal{H}$. Then
\begin{itemize}\item [$(i)$] $\mathcal{W}$ is called an orthonormal fusion system or simply a orthonormal $f$-system for $\mathcal{H}$, if
$\{\pi_{W_j}\}_{j\in J}$ is a $f$-biorthogonal sequence of $\mathcal{W}$. That is
$$\pi_{W_i}g_j=\delta_{ij}g_j\;\;\;\;\;\forall i,j\in J, g_j\in W_j.$$\item [$(ii)$] $\mathcal{W}$ is called an orthonormal $f$-basis
for $\mathcal{H}$, if it is a complete orthonormal $f$-system for $\mathcal{H}$.\end{itemize}\end{definition}\begin{example}
Let $\{e_i\}_{i\in\mathbb{N}}$ be an orthonormal basis for $\mathcal{H}$. Then\begin{itemize}\item [$(i)$] If for each $j\in\mathbb{N}$
define the subspace $W_j\subset\mathcal{H}$ by $$W_j=\Span\{e_{2j-1}+e_{2j}\}\;\;\;\;\text{and}\;\;\;\; \pi_{W_j}(f)=\frac{1}{2}<f, e_{2j-1}+e_{2j}>
(e_{2j-1}+e_{2j}).$$ Then it is easily checked that $\{W_j\}_{j\in\mathbb{N}}$ is an orthonormal $f$-system for $\mathcal{H}$. But it is not an orthonormal
$f$-basis for $\mathcal{H}$. \item [$(ii)$] If for each $j\in\mathbb{N}$ define the subspace $W_j\subset\mathcal{H}$ by $$W_j=\Span\{e_{2j-1}, e_{2j}\}
\;\;\;\;\text{and}\;\;\;\; \pi_{W_j}(f)=<f, e_{2j-1}>e_{2j-1}+<f, e_{2j}>e_{2j}.$$ Then $\{W_j\}_{j\in\mathbb{N}}$ is an orthonormal $f$-basis for
$\mathcal{H}$.\end{itemize}\end{example} The following theorem shows that every orthonormal $f$-basis is a Bessel and Hilbert $f$-basis for $\mathcal{H}$.
\begin{theorem} Let $\{W_j\}_{j\in J}$ be an orthonormal $f$-system for $\mathcal{H}$, then the series $\sum_{j\in J}g_j$ converges if and only if
$\{g_j\}_{j\in J}\in\big(\sum_{j\in J}\oplus W_j\big)_{\ell^2}$ and in this case $$\big\|\sum_{j\in J}g_j\big\|^2=\sum_{j\in J}\|g_j\|^2.$$
\end{theorem}\begin{proof} For every finite subset $F\subset J$ we have $$\big\|\sum_{j\in F}g_j\big\|^2=\sum_{j\in F}\sum_{i\in F}<\pi_{W_i}g_j, g_i>=
\sum_{j\in F}\sum_{i\in F}<\delta_{ij}g_j, g_i>=\sum_{j\in F}\|g_j\|^2.$$ From this the result follows.\end{proof}
\begin{theorem} $($Bessel's inequality$)$ If $\{W_j\}_{j\in J}$ is an orthonormal $f$-system for $\mathcal{H}$. Then \begin{align*}
\sum_{j\in J}\|\pi_{W_j}f\|^2\leq\|f\|^2\;\;\;\;\;\;\text{for all}\;\; f\in\mathcal{H}.\end{align*}\end{theorem}\begin{proof} Let $f\in\mathcal{H}$.
Fix $F\subset J$ with $|F|<\infty$. Then By Theorem 3.3 we have \begin{align*}\Big\|f-\sum_{j\in F}g_j\Big\|^2&=
\|f\|^2-\sum_{j\in F}<\pi_{W_j}f, g_j>-\sum_{j\in F}<g_j, \pi_{W_j}f>
+\sum_{j\in F}\|g_j\|^2\\&= \|f\|^2-\sum_{j\in F}\|\pi_{W_j}f\|^2+\sum_{j\in F}\|\pi_{W_j}f-g_j\|^2\end{align*} for arbitrary vectors $g_j\in W_j$.
In particular, if $g_j=\pi_{W_j}f$, then $$\big\|f-\sum_{j\in F}\pi_{W_j}f\big\|^2=\|f\|^2-\sum_{j\in F}\|\pi_{W_j}f\|^2.$$
From this we have $\sum_{j\in F}\|\pi_{W_j}f\|^2\leq\|f\|^2$, which implies that $\sum_{j\in J}\|\pi_{W_j}f\|^2\leq\|f\|^2$.\end{proof}
\begin{corollary} If $\{W_j\}_{j\in J}$ is an orthonormal $f$-system for $\mathcal{H}$, then for all $f\in\mathcal{H}$ the series $\sum_{j\in J}\pi_{W_j}f$
convergent and $$\big\|f-\sum_{j\in J}\pi_{W_j}f\big\|^2\leq\big\|f-\sum_{j\in J}g_j\big\|^2$$ for all $\{g_j\}_{j\in J}\in\big(\sum_{j\in J}\oplus W_j\big)_{\ell^2}$.\end{corollary}
\begin{theorem} Let $\mathcal{W}=\{W_j\}_{j\in J}$ be an orthonormal $f$-system for $\mathcal{H}$. Then the following conditions are equivalent:
\begin{itemize}\item [$(i)$] $\mathcal{W}$ is an orthonormal $f$-basis for $\mathcal{H}$.\item [$(ii)$] $f=\sum_{j\in J}\pi_{W_j}f\;\;\;\;\forall f\in \mathcal{H}.$ \item [$(iii)$] $\|f\|^2=\sum_{j\in J}\|\pi_{W_j}f\|^2\;\;\;\;\forall f\in \mathcal{H}.$
\item [$(iv)$] If $\pi_{W_j}f=0$ for all $j\in J$, then $f=0$.\end{itemize}
\end{theorem}\begin{proof} The implication $(i)\Rightarrow (ii)$ follows immediately from Corollary 3.5. The implications $(ii)\Rightarrow (iii)
\Rightarrow (iv)$ are obvious. To prove $(iv)\Rightarrow (i)$ suppose that $f\perp\overline{\Span}\{W_j\}_{j\in J}$, then for every $j\in J$
we have $\|\pi_{W_j}f\|^2=<f, \pi_{W_j}f>=0$ which implies that $f=0$. Therefore $\mathcal{H}=\overline{\Span}\{W_j\}_{j\in J}.$\end{proof}
\begin{theorem} Let $\{(W_j, P_{W_j})\}_{j\in J}$ be a $f$-basis system for $\mathcal{H}$ and let $T:\mathcal{H}\rightarrow\mathcal{K}$ be a bounded
invertible operator such that $V_j=TW_j$ and $Q_{V_j}=TP_{W_j}T^{-1}$ for all $j\in J$. Then $\{(V_j, Q_{V_j})\}_{j\in J}$
is a $f$-basis system for $\mathcal{K}$.\end{theorem}\begin{proof} Suppose that $f\in\mathcal{K}$, then we can write $f=Tg$ for some $g\in\mathcal{H}$.
By hypothesis $g$ has an unique expansion to form $g=\sum_{j\in}g_j$ for some sequence $\{g_j:\;g_j\in W_j\}_{j\in J}$ which implies
that $f$ has an unique expansion of the form $f=\sum_{j\in J}f_j$ with $f_j=Tg_j$ for all $j\in J$. We also have\begin{align*} Q_{V_i}f_j=
TP_{W_i}T^{-1}f_j=T(\delta_{ij}T^{-1}f_j)=\delta_{ij}f_j\end{align*} for arbitrary sequence $\{f_j:\;f_j\in V_j\}_{j\in J}$. From this the result
follows.\end{proof}\begin{definition} Let $\{W_j\}_{j\in J}$ be a sequence of closed subspaces of $\mathcal{H}$, then $\mathcal{W}=\{W_j\}_{j\in J}$
is called a Riesz fusion basis or simply Riesz $f$-basis for $\mathcal{H}$ if there is an orthonormal $f$-basis $\{V_j\}_{j\in J}$ for $\mathcal{H}$
and a bounded invertible linear operator $T$ on $\mathcal{H}$ such that $TV_j=W_j$ for all $j\in J$. By Theorem 3.7 if $\{P_{W_j}\}_{j\in J}$
is $f$-dual sequence of $\{W_j\}_{j\in J}$, then $P_{W_j}=T\pi_{V_j}T^{-1}$ for all $j\in J$.\end{definition}
\begin{example} Let $\{f_j\}_{j\in J}=\{T(e_j)\}_{j\in J}$ be a Riesz basis for $\mathcal{H}$ and let $W_j=\Span\{f_j\}$ for all $j\in J$. Then
$\mathcal{W}=\{W_j\}_{j\in J}$ is a Riesz $f$-basis for $\mathcal{H}$. Because if $V_j=\Span\{e_j\}$, then $\{V_j\}_{j\in J}$ is an orthonormal $f$-basis and
$W_j=TV_j$ for all $j\in J$.\end{example}\begin{corollary} If $\{(W_j, P_{W_j})\}_{j\in J}$ is a Riesz $f$-basis system for $\mathcal{H}$. Then
$$0<\inf_{j\in J}\|P_{W_j}\|\leq\sup_{j\in J}\|P_{W_j}\|<\infty.$$ \end{corollary}\begin{proof} According to the definition we can write $\{W_j\}_{j\in J}
=\{TV_j\}_{j\in J}$, where $T$ is a bounded invertible operator and $\{V_j\}_{j\in J}$ is an orthonormal $f$-basis for $\mathcal{H}$. Since for every $j\in J$
we have $$\|T^{-1}\|^{-1}\|T\|^{-1}\leq\|P_{W_j}\|\leq\|T\|\|T^{-1}\|.$$ Hence from this the result follows.\end{proof}
\begin{proposition} Let $\{W_j\}_{j\in J}=\{TV_j\}_{j\in J}$ be a Riesz $f$-basis for $\mathcal{H}$ and let $\{f_j:\;f_j\in V_j\}_{j\in J}$ and
$\{g_j:\;g_j\in W_j\}_{j\in J}$ be sequences in $\mathcal{H}$ such that $g_j=Tf_j$ for all $j\in J$. Then $\sum_{j\in J}g_j$ is convergent
if and only if $\sum_{j\in J}f_j$ is convergent.\end{proposition}\begin{proof} This follows immediately from the fact that for each finite subset
$F\subset J$ we have $$\|T^{-1}\|^{-1}\big\|\sum_{j\in F}f_j\big\|\leq\big\|\sum_{j\in F}g_j\big\|\leq\|T\|\big\|\sum_{j\in F}f_j\big\|.$$\end{proof}
\begin{definition} A family of bounded operators $\{T_j\}_{j\in J}$ on $\mathcal{H}$ is called a resolution of the identity on $\mathcal{H}$, if for all
$f\in\mathcal{H}$ we have $f=\sum_{j\in J}T_jf$.\end{definition} The following result shows another way to obtain a resolution of the identity from
a Riesz $f$-basis.\begin{proposition} Let $\{W_j\}_{j\in J}=\{TV_j\}_{j\in J}$ be a Riesz $f$-basis for $\mathcal{H}$. Then
\begin{itemize}\item [$(i)$] $\{S_j\}_{j\in J}$ defined by $S_j=T^{-1}\pi_{V_j}T,\;j\in J$ is a resolution of the identity on $\mathcal{H}$.\item [$(ii)$]
$\{U_j\}_{j\in J}$ defined by $U_j=T^{*}\pi_{V_j}(T^*)^{-1},\;j\in J$ is a resolution of the identity on $\mathcal{H}$.\item [$(iii)$]
$\{R_j\}_{j\in J}$ defined by $R_j=(T^*)^{-1}\pi_{V_j}T^*,\;j\in J$ is a resolution of the identity on $\mathcal{H}$.\end{itemize}
\end{proposition}\begin{proof} This follows immediately from the definition.\end{proof} To check Riesz $f$-baseness of a family of
closed subspaces $\{W_i\}_{i\in I}$, we derive the following useful characterization.
\begin{theorem} Let $\{W_j\}_{j\in J}$ be a family of closed subspaces of $\mathcal{H}$. Then the following statements are equivalent.
\begin{itemize}\item [$(i)$] $\{W_j\}_{j\in J}$ is a Riesz $f$-basis for $\mathcal{H}$.\item [$(ii)$] There is an equivalent inner product on $\mathcal{H},$
with respect to which the sequence $\{W_j\}_{j\in J}$ becomes an orthonormal $f$-basis for $\mathcal{H}$.\item [$(iii)$] The family $\{W_j\}_{j\in J}$ is
complete for $\mathcal{H}$ and there exist positive constants $A, B$ such that for any finite subset $F\subset J$ and arbitrary vectors
$g_j\in W_j$, we have $$A\sum_{j\in F}\|g_j\|^2\leq\big\|\sum_{j\in F}g_j\big\|^2\leq
B\sum_{j\in F}\|g_j\|^2.$$\end{itemize}\end{theorem}\begin{proof}
$(i)\Rightarrow(ii)$ Assume that $\{W_j\}_{j\in J}$ is a Riesz $f$-basis, and write it in the form $\{TV_j\}_{j\in J}$ as in the definition. Define
a new inner product $<. , .>_{T}$ on $\mathcal{H}$ by $$<f, g>_{T}=<T^{-1}f, T^{-1}g>\;\;\;\;\;\forall f,g\in\mathcal{H}.$$ If $\|.\|_{T}$ is the norm
generated by this inner product, then for all $f\in\mathcal{H}$ we have $$\|T\|^{-1}\|f\|\leq\|f\|_{T}\leq\|T^{-1}\|\|f\|,$$ which implies that the new
inner product is equivalent to the original one. Let $\{P_{W_j}\}_{j\in J}$ be $f$-dual sequence of $\{W_j\}_{j\in J}$, then we have\begin{align*}
<P_{W_j}f, g>_{T}&=<T\pi_{V_j}T^{-1}f, g>_{T}=<\pi_{V_j}T^{-1}f, T^{-1}g>\\&=<T^{-1}f, \pi_{V_j}T^{-1}g>=<f, P_{W_j}g>_{T}.\end{align*} It follows that
$P_{W_j}:\mathcal{H}\rightarrow W_j$ is an orthogonal projection with respect to $<. , .>_{T}$, thus $\{W_j\}_{j\in J}$ is an orthonormal $f$-basis with
respect to inner product $<. , .>_{T}$.\par$(ii)\Rightarrow(iii)$ Suppose that $<. , .>_1$ is an equivalent inner product on $\mathcal{H}$ and let
$\{W_j\}_{j\in J}$ be an orthonormal $f$-basis with respect to $<. , .>_1$. Then there exist positive constants $m, M$ such that $$m\|f\|\leq\|f\|_1
\leq M\|f\|\;\;\;\;\;\forall f\in\mathcal{H}.$$ For any finite subset $F\subset J$ and arbitrary sequence $\{g_j:\;g_j\in W_j\}_{j\in J}$ we obtain
\begin{align*}\frac{m^2}{M^2}\sum_{j\in F}\|g_j\|^2&\leq\frac{1}{M^2}\sum_{j\in F}\|g_j\|_{1}^2=\frac{1}{M^2}\big\|\sum_{j\in F}g_j\big\|_{1}^2\\&\leq
\big\|\sum_{j\in F}g_j\big\|^2\leq\frac{1}{m^2}\big\|\sum_{j\in F}g_j\big\|_{1}^2\\&=\frac{1}{m^2}\sum_{j\in F}\|g_j\|_{1}^2\leq\frac{M^2}{m^2}
\sum_{j\in F}\|g_j\|^2.\end{align*}\par $(iii)\Rightarrow(i)$ Let $\{V_j\}_{j\in J}$ be an arbitrary orthonormal $f$-basis for $\mathcal{H}$.
Define the mapping $$T:\mathcal{H\rightarrow\mathcal{H}},\;\;\;\;\;\;\text{with}\;\;\;\;\; T\pi_{V_j}f=\pi_{W_j}f\;\;\;\;\;\forall
f\in\mathcal{H},\;\;j\in J.$$ For every sequence $\{f_j:\;f_j\in V_j\}_{j\in J}$ we have \begin{align*}
\big\|T(\sum_{j\in J}f_j)\big\|^2=\big\|\sum_{j\in J}\pi_{W_j}f_j\big\|^2\leq B\sum_{j\in J}
\|\pi_{W_j}f_j\|^2\leq B\big\|\sum_{j\in J}f_j\big\|^2,\end{align*} it follows that $T$ is a bounded linear operator on $\mathcal{H}$.
Similarly, define the mapping $$S:\mathcal{H}\rightarrow\mathcal{H},\;\;\;\;\;\;\text{with}\;\;\;\;\; S\pi_{W_j}f=\pi_{V_j}f
\;\;\;\;\;\forall f\in\mathcal{H},\;\;j\in J.$$ We also obtain\begin{align*}\big\|S(\sum_{j\in J}g_j)\big\|^2=
\big\|\sum_{j\in J}\pi_{V_j}g_j\big\|^2\leq \sum_{j\in J}\|g_j\|^2\leq \frac{1}{A}\big\|\sum_{j\in J}g_j\big\|^2,
\end{align*} for any sequence $\{g_j:\;g_j\in W_j\}_{j\in J}$. Since $\{W_j\}_{j\in J}$ is complete thus $S$ is also a linear bounded operator on
$\mathcal{H}$ and $TS=ST=Id_{\mathcal{H}}$. Hence $T$ is invertible and $TV_j=W_j$ for all $j\in J$.\end{proof}\begin{corollary} $\{W_j\}_{j\in J}$
is a Riesz $f$-basis for $\mathcal{H}$ if and only if it is both a Bessel $f$-basis and a Hilbert $f$-basis for $\mathcal{H}$.\end{corollary}
\begin{corollary} Let $\{W_j\}_{j\in J}$ be a family of closed subspaces of $\mathcal{H}$ and for each $j\in J$ let $\{e_{ij}\}_{i\in I_j}$ be an orthonormal
basis for each subspace $W_j$. Then the following conditions are equivalent.\begin{enumerate}\item [$(i)$] $\{W_j\}_{j\in J}$ is a Riesz $f$-basis for $\mathcal{H}$.\item [$(ii)$] $\{e_{ij}\}_{j\in J, i\in I_j}$ is a Riesz basis for $\mathcal{H}$.\end{enumerate}\end{corollary}\begin{proof}
$(i)\Rightarrow (ii)$ Assume that $\{W_j\}_{j\in J}$ is a Riesz $f$-basis, then by Theorem 3.14 there exist constants $A, B>0$ such that
for any finite subset $F\subset J$ and arbitrary vectors $g_j\in W_j$ $$A\sum_{j\in F}\|g_j\|^2\leq\big\|\sum_{j\in F}g_j\big\|^2\leq
B\sum_{j\in F}\|g_j\|^2.$$ Let $F\subset J$ and $G_j\subset I_j$ be arbitrary finite subsets, then for all arbitrary scalars
$\{c_{ij}\}_{j\in F, i\in G_j}$ we have \begin{align*} A\sum_{j\in F}\sum_{i\in G_j}|c_{ij}|^2&=
A\sum_{j\in F}\big\|\sum_{i\in G_j}c_{ij}e_{ij}\big\|^2\leq\big\|\sum_{j\in F}\sum_{i\in G_j}c_{ij}e_{ij}\big\|^2
\\&\leq B\sum_{j\in F}\big\|\sum_{i\in G_j}c_{ij}e_{ij}\big\|^2=B\sum_{j\in F}\sum_{i\in G_j}|c_{ij}|^2.\end{align*} By [4, Theorem 3.6.6] $\{e_{ij}\}
_{j\in J, i\in I_j}$ is a Riesz basis for $\mathcal{H}$.\par $(ii)\Rightarrow (i)$ Suppose that $\{e_{ij}\}_{j\in J, i\in I_j}$ is a Riesz basis for
$\mathcal{H}$, by [4, Theorem 3.6.6] there exist constants $A, B>0$ such that for any finite subsets $F\subset J, G_j\subset I_j$ and arbitrary scalars
$\{c_{ij}\}_{j\in F, i\in G_j}$ \begin{align*} A\sum_{j\in F}\sum_{i\in G_j}|c_{ij}|^2\leq\big\|\sum_{j\in F}\sum_{i\in G_j}c_{ij}e_{ij}\big\|^2
\leq B\sum_{j\in F}\sum_{i\in G_j}|c_{ij}|^2.\end{align*} Assume that $\{g_j:\;g_j\in W_j\}_{j\in J}$ is an arbitrary sequence, then
we have\begin{align*} A\sum_{j\in F}\|g_j\|^2&=A\sum_{j\in F}\sum_{i\in I_j}|<g_j, e_{ij}>|^2\leq\big\|\sum_{j\in F}\sum_{i\in I_j}
<g_j, e_{ij}>e_{ij}\big\|^2\\&\leq B\sum_{j\in F}\sum_{i\in I_j}|<g_j, e_{ij}>|^2=B\sum_{j\in F}\|g_j\|^2.\end{align*} Now we observe that $$\big\|\sum_{j\in F}g_j\big\|^2=\big\|\sum_{j\in F}\sum_{i\in I_j}<g_j, e_{ij}>e_{ij}\big\|^2.$$ By Theorem 3.14 the result follows at once.\end{proof}
Let us introduce a result due to Gavruta [8, Theorem 2.4].\begin{theorem} Let $\{(V_j, \alpha_j)\}_{j\in J}$ be a fusion frame with fusion frame bounds $C, D$.
If $T:\mathcal{H}\rightarrow\mathcal{H}$ is an invertible operator, then $\{(TV_j, \alpha_j)\}_{j\in J}$ is a fusion frame with fusion frame bounds
$$C\|T\|^{-2}\|T^{-1}\|^{-2}\;\;\;\;\;\text{and}\;\;\;\;\;D\|T\|^{2}\|T^{-1}\|^{2}.$$\end{theorem}\begin{corollary} If $\{W_j\}_{j\in J}$ is
a Riesz $f$-basis for $\mathcal{H}$ then $\{(W_j , 1)\}_{j\in J}$ is a $1$-uniform fusion frame with fusion frame bounds $$\|T\|^{-2}\|T^{-1}\|^{-2}
\;\;\;\;\;\text{and}\;\;\;\;\;\|T\|^{2}\|T^{-1}\|^{2}.$$\end{corollary} A fusion frame $\{(W_j,\alpha_j)\}_{j\in J}$ is called exact, if it ceases
to be a fusion frame whenever anyone of its element is deleted.\begin{theorem} Let $\{W_j\}_{j\in J}$ be a
Riesz $f$-basis for $\mathcal{H}$, then $\{(W_j, 1)\}_{j\in J}$ is a $1$-uniform exact fusion frame for
$\mathcal{H}$. But the opposite implication is not valid.\end{theorem}\begin{proof} Let $\{e_{ij}\}_{i\in I_j}$ be an orthonormal basis for $W_j$ for all
$j\in J$. Then by Theorem 3.16, $\{e_{ij}\}_{j\in J, i\in I_j}$ is a Riesz basis for $\mathcal{H}$ and hence it is an exact frame. Now by [3, Lemma 4.5]
$\{(W_j, 1)\}_{j\in J}$ is a $1$-uniform exact fusion frame for $\mathcal{H}$. For the opposite implication is not valid suppose that $\{e_i\}_{i\in\mathbb{Z}}$
is an orthonormal basis for $\mathcal{H}$ and define the subspaces $W_1$ and $W_2$ by $$W_1=\overline{\Span}\{e_i\}_{i\geq 0}\;\;\;\;\text{and}\;\;\;\;
W_2=\overline{\Span}\{e_i\}_{i\leq 0}.$$ Then $\{(W_1, 1) , (W_2, 1)\}$ is a $1$-uniform exact fusion frame but is not a Riesz $f$-basis for $\mathcal{H}$.
\end{proof}

\section{The Stability of $f$-bases under perturbations}

The stability of bases is important in practice and is therefore studied widely by many authors, e.g., see [13]. In this section we study the stability of
$f$-bases for a Hilbert space $\mathcal{H}$. First we generalized a result of Paley-Wiener [13] to the situation of $f$-basis.
\begin{theorem} Let $\{W_j\}_{j\in J}$ be a $f$-basis for $\mathcal{H}$ and suppose that $\{V_j\}_{j\in J}$ is a family of closed subspaces of $\mathcal{H}$
such that $$\big\|\sum_{j\in F}(g_j-f_j)\big\|\leq\lambda\big\|\sum_{j\in F}g_j\big\|$$ for some constant $\lambda,\;0\leq\lambda<1$ and any finite subset
$F\subset J$ and arbitrary vectors $g_j\in W_j, f_j\in V_j$. Then $\{V_j\}_{j\in J}$ is a $f$-basis for $\mathcal{H}$.\end{theorem}\begin{proof}
For all sequences $\{g_j:\;g_j\in W_j\}_{j\in J}$ and $\{f_j:\;f_j\in V_j\}_{j\in J}$ it follows by assumption that the series $\sum_{j\in J}(g_j-f_j)$ is
convergent whenever the series $\sum_{j\in J}g_j$ is convergent. Define the mapping $$T:\mathcal{H\rightarrow\mathcal{H}},\;\;\;\;\;\;
\text{with}\;\;\;\;\; T\pi_{W_j}f=\pi_{W_j}f-\pi_{V_j}f\;\;\;\;\;\forall f\in\mathcal{H},\;\;j\in J.$$
Let $f=\sum_{j\in J}g_j$ be an arbitrary element of $\mathcal{H}$ for some sequence $\{g_j:\;g_j\in
W_j\}_{j\in J}$, then \begin{align*}\|Tf\|=\big\|\sum_{j\in J}(g_j-\pi_{V_j}g_j)\big\|\leq \lambda\|f\|.\end{align*} It follows that $T$ is a bounded
linear operator on $\mathcal{H}$ and $\|T\|\leq\lambda<1$. Thus the operator $Id_{\mathcal{H}}-T$ is invertible. Since $(Id_{\mathcal{H}}-T)W_j=V_j$ for every
$j\in J$ hence by Theorem 3.7 the result follows.\end{proof} A family of subspaces $\{W_j\}_{j\in J}$ of $\mathcal{H}$ is called minimal, if $W_i\cap
\overline{\Span}_{j\in J, i\neq j}\{W_j\}=\{0\}$ for all $i\in J$.
\begin{proposition} Let $\{W_j\}_{j\in J}$ be a sequence of closed subspaces of $\mathcal{H}$. Then\begin{itemize}\item [$(i)$] $\{W_j\}_{j\in J}$
has a $f$-biorthogonal sequence, if and only if it is minimal.\item [$(ii)$] The $f$-biorthogonal sequence of $\{W_j\}_{j\in J}$ is unique if
and only if it is complete.\end{itemize}\end{proposition}\begin{proof}
For the proof of $(i)$ suppose that $\{P_{W_j}\}_{j\in J}$ is a $f$-biorthogonal sequence of $\{W_j\}_{j\in J}$ and let
$f\in W_i\cap\overline{\Span}\{W_j\}_{j\in J,\atop{j\neq i}}$ for any given $i\in J$. Then there exists a sequence
$\{g_j:\;g_j\in W_j\}_{j\in J}$ such that $f=g_i=\sum_{j\in J,\atop{j\neq i}}g_j$. We also have \begin{align*} f=g_i=P_{W_i}g_i
=\sum_{j\in J,\atop{j\neq i}}P_{W_i}g_j=\sum_{j\in J,\atop{j\neq i}}\delta_{ij}g_j=0,\end{align*} which implies that $f=0$. That is,
$\{W_j\}_{j\in J}$ is minimal. For the opposite implication in $(i)$, suppose that $\{W_j\}_{j\in J}$ is minimal, and let
$\mathcal{H}_{0}=\overline{\Span}\{W_j\}_{j\in J}$, then $\{W_j\}_{j\in J}$ is a $f$-basis for $\mathcal{H}_{0}$.
Let $\{P'_{W_j}\}_{j\in J}$ be $f$-dual sequence of $\{W_j\}_{j\in J}$ for $\mathcal{H}_{0}$. If we define $P_{W_j}=P'_{W_j}\pi_{\mathcal{H}_{0}}$
for all $j\in J$. Then $\{P_{W_j}\}_{j\in J}$ is a $f$-biorthogonal sequence for $\{W_j\}_{j\in J}$.\par $(ii)$ Let $\{P_{W_j}\}_{j\in J}$ be a
$f$-biorthogonal sequence of $\{W_j\}_{j\in J}$. If $\{W_j\}_{j\in J}$ is not complete, then the sequence $\{Q_{W_j}\}_{j\in J}$ defined by $Q_{W_j}=P_{W_j}+P_{W_j}(Id_{\mathcal{H}}-\pi_{\mathcal{H}_{0}})$ for all $j\in J$, is a $f$-biorthogonal sequence for
$\{W_j\}_{j\in J}$. For the other implication in $(ii)$, assume that $\{W_j\}_{j\in J}$ is complete. If $\sum_{j\in J}g_j=0$
for any given sequence $\{g_j:\;g_j\in W_j\}_{j\in J}$, then for every $i\in J$ we have \begin{align*} g_i=\sum_{j\in J}\delta_{ij}g_j
=\sum_{j\in J}P_{W_i}g_j=P_{W_i}(\sum_{j\in J}g_j)=0.\end{align*} This shows that $\{W_j\}_{j\in J}$ is a $f$-basis for $\mathcal{H}$.
Now the conclusion follows from Proposition 2.10.\end{proof}
\begin{theorem} Let $\{(W_j, P_{W_j})\}_{j\in J}$ be a $f$-basis system for $\mathcal{H}$,
and suppose that $\{V_j\}_{j\in J}$ is a family of closed subspaces of $\mathcal{H}$. If $\{Q_{V_j}\}_{j\in J}$ is a $f$-biorthogonal sequence of
$\{V_j\}_{j\in J}$ such that $$\big\|\sum_{j\in F}(P_{W_j}f-Q_{V_j}f)\big\|\leq\lambda\big\|\sum_{j\in F}P_{W_j}f\big\|\;\;\;\;\;\forall f\in\mathcal{H}$$
for some constant $\lambda,\;0\leq\lambda<1$ and any finite subset $F\subset J$. Then $\{V_j\}_{j\in J}$ is a $f$-basis for $\mathcal{H}$.\end{theorem}
\begin{proof} Define the mapping $$T:\mathcal{H\rightarrow\mathcal{H}},\;\;\;\;\;\;\text{with}\;\;\;\;\; TP_{W_j}f=P_{W_j}f-Q_{V_j}
f\;\;\;\;\;\forall f\in\mathcal{H},\;\;j\in J.$$ Then the operator $Id_{\mathcal{H}}-T$ is invertible and $(Id_{\mathcal{H}}-T)W_j=V_j$ for every
$j\in J$ hence by Theorem 3.7 the result follows.\end{proof}

\vspace{5mm}
{\bf Acknowledgements:} The author express his gratitude to the referee for carefully reading of the manuscript and
giving useful comments.


\begin{thebibliography}{10}

\bibitem{} M. S. Asgari, New characterizations of fusion frames (frames of subspaces), {\it Proc. Indian
Acad. Sci. (Math. Sci.)} {\bf 119} no. 3 (2009), 1-14.
\bibitem{} M. S. Asgari, On the stability of fusion frames (frames of subspaces),
{\it Acta Math. Sci. Ser. B}, {\bf 31}(4), (2011), 1633-1642.
\bibitem{} P. G. Casazza and G. Kutyniok, Frames of subspaces, in "Wavelets, Frames and Operator Theory"
(College Park, MD, 2003), {\it Contemp. Math. } {\bf 345}, {\it Amer. Math. Soc.}
Providence, RI, 2004, 87-113.
\bibitem{} O. Christensen, An Introduction to frames and Riesz Bases, Birkhauser, Boston, 2003.
\bibitem{} I. Daubechies, A. Grossmann and Y. Meyer, Painless nonorthogonal
expansions, {\it J. Math. Phys. } {\bf 27}(1986), 1271-1283.
\bibitem{} R. J. Duffin and A. C. Schaeffer, A class of nonharmonic Fourier
series, {\it Trans. Amer. Math. Soc.} {\bf 72},{\bf (2)}, (1952), 341-366.
\bibitem{} M. Fornasier, Quasi-orthogonal decompositions of structured
frames, {\it J. Math. Anal. Appl. } {\bf 289} (2004), 180-199.
\bibitem{} P. Gavruta, On the duality of fusion frames, {\it J. Math.
Anal. Appl. } {\bf 333} (2007), 871-879.
\bibitem{} J.R. Holub, Pre-frame operators, Besselian frames and near-Riesz bases in
Hilbert spaces, {\it Proc. Amer. Math. Soc.}
{\bf 122} (1994) 779-785.
\bibitem{} V. Kaftal, D. R. Larson and Sh. Zhang, Operator-valued frames, {\it Trans. Amer. Math. Soc.}
{\bf 361} (2009), 6349-6385.
\bibitem{} W. Rudin, Functional Analysis, McGraw–Hill. Inc, New York, (1991).
\bibitem{} W. Sun, G-frames and G-Riesz bases, {\it J. Math. Anal.
Appl.} (2006), {\bf 322}, 437-452.
\bibitem{} R. Young, An Introduction to Nonharmonic Fourier Series, Academic Press, New York, 2001.
















\end{thebibliography}
\end{document}